\documentclass[8pt]{article}
\usepackage{fullpage}
\usepackage{amsmath}
\usepackage{amssymb}
\usepackage{bm}
\usepackage{xspace}
\usepackage[USenglish]{babel}
\usepackage{microtype}
\usepackage{color}
\usepackage[normalem]{ulem}
\usepackage{amsmath}
\usepackage{amssymb}
\usepackage{amsthm}
\usepackage{enumerate}
\usepackage{tikz}
\usepackage{mdframed}
\usepackage{enumitem}
\usepackage{algorithm}
\usepackage{algpseudocode}
\usepackage{url}
\usepackage{hyperref}
\usepackage[nameinlink,capitalize]{cleveref}
\usepackage{natbib}   
\crefname{lemma}{lemma}{lemmas}
\Crefname{lemma}{Lemma}{Lemmas}
\crefname{thm}{theorem}{theorems}
\Crefname{thm}{Theorem}{Theorems}

\newtheorem{theorem}{Theorem}

\newtheorem{corollary}{Corollary}[section]
\newtheorem{lemma}{Lemma}[section]
\newtheorem{prop}{Proposition}[section]

\newcommand*{\Exp}{\mathbb{E}}
\newcommand*{\Var}{\textrm{Var}}

\newcommand*{\pall}{p_{\text{all}}}
\newcommand*{\hpall}{\hat{p}_{\text{all}}}
\newcommand*{\muall}{\mu_{\text{all}}}
\newcommand*{\hmuall}{\hat{\mu}_{\text{all}}}

\newcommand{\ceil}[1]{\left\lceil #1 \right\rceil}
\newcommand*{\algname}{\textsc{ABae}\xspace}

\newcommand{\minihead}[1]{{\vspace{.45em}\noindent\textbf{#1.} }}

\begin{document}
\title{Proof: Accelerating Approximate Aggregation Queries with Expensive Predicates}
\author{Daniel Kang$^{*, \dagger}$, John Guibas$^{*, \dagger}$, Peter Bailis$^\dagger$, Tatsunori Hashimoto$^\dagger$, Yi Sun$^\ddagger$, Matei Zaharia$^\dagger$ \\ Stanford University$^\dagger$, University of Chicago$^\ddagger$}
\maketitle
\begin{abstract}
    \noindent Given a dataset $\mathcal{D}$, we are interested in computing the mean of a subset of $\mathcal{D}$ which matches a predicate. \algname leverages stratified sampling and proxy models to efficiently compute this statistic given a sampling budget $N$. In this document, we theoretically analyze \algname and show that the MSE of the estimate decays at rate $O(N_1^{-1} + N_2^{-1} + N_1^{1/2}N_2^{-3/2})$, where $N=K \cdot N_1+N_2$ for some integer constant $K$ and $K \cdot N_1$ and $N_2$ represent the number of samples used in Stage 1 and Stage 2 of \algname respectively.  Hence, if a constant fraction of the total sample budget $N$ is allocated to each stage, we will achieve a mean squared error of $O(N^{-1})$ which matches the rate of mean squared error of the optimal stratified sampling algorithm given a priori knowledge of the predicate positive rate and standard deviation per stratum.
\end{abstract}

\begin{table}[htbp]
\centering
\begin{tabular}{ll}
Symbol & Description \\ \hline
$\mathcal{D}$   & dataset of records \\
$O(x)$   & predicate function \\
$f(x)$   & statistic function \\
$K$   & number of strata \\
$S$   & set of $K$ equal sized stratum formed from $\mathcal{D}$ \\
$X_{k,i}$   & subgaussian r.v. in $S_k$ and the i'th i.i.d predicate-matching sample  \\
$p_k$   & predicate positive rate in stratum $k$ \\
$\mu_k$   & $E[X_k]$  \\
$\sigma^2_k$   & $Var[X_k]$ \\
$C_{k}^{(\mu)}$   & upper bound on $\mu_k$ \\
$C_{k}^{(\sigma)}$   & upper bound on $\sigma_k$ \\
$p_{\text{all}}$   & $\sum_i^K p_k$ \\
$w_k$   & $p_k/\pall$ \\
$\muall$   & $E[X]$, also given by $\sum_{k=1}^K w_k\mu_k$ \\
$N_1$   & sampling budget we allocate per stratum during $1$ \\
$N_2$   & sampling we allocate during $2$ in total\\
$N$   & total sampling budget ($K\cdot N_1+N_2=N)$\\
$T_{k}$   & fraction of our $N_2$ samples we allocate to a stratum during $2$ \\
$B_k^{(1)}$   & number of predicate-matching samples we get in Stage 1 \\
$B_k^{(2)}$   & number of predicate-matching samples we get in Stage 2 \\
$p_*$   & cutoff value used when we split cases where $p_k$ is small and large \\
$F_k$   & we say that $B_k^{(2)} \geq F_k$ with high probability \\
\end{tabular}
\caption{Summary of notation.}\label{nottable}
\label{table:notation}
\end{table}
    
\section{Setup}
Given a dataset $\mathcal{D} = \{ x_1, ..., x_{|D|} \}$, a statistic $f : \mathcal{D} \xrightarrow{} \mathbb{R}$, and a predicate $O : \mathcal{D} \xrightarrow{} \{0, 1\}$, we want to estimate $\Exp_{x \sim \mathcal{D}}[f(x) | \ O(x)=1]=\muall$ given a sampling budget $N$, which limits the number of times the predicate $O(x)$ can be computed. We measure the quality of an estimator for $\muall$ by measuring the mean squared error (MSE) or $\Exp[(\hmuall - \muall)^2]$.

\vspace{3mm}

To estimate this mean efficiently, we use stratified sampling. In stratified sampling, the data is split into $K$ strata or disjoint subsets of the dataset. In other words, given a dataset $\mathcal{D}$, we stratify it into equal-sized strata $S_1, ..., S_K$ such that they are disjoint and span the entirety of $\mathcal{D}$. Given these strata, denote $\mu_k$, $p_k$, $\sigma_k$ as the per-stratum mean of records which match the predicate, the per-stratum predicate positive rate, and the per-stratum standard deviation. Then, $\muall=\sum_{k=1}^K \mu_kp_k  / \sum_{k=1}^K p_k$. 

\vspace{3mm}

We make the following assumptions on our dataset. First, we assume that the
distribution of the data we sample from is subgaussian. This is so that we can
obtain concentration inequalities on estimates of variables such as $\mu_k$ and
$\sigma_k$. We represent the samples from a stratum by a subgaussian random
variable $X_{k, i}$ where $k$ denotes the stratum and $i$ denotes the i'th
independent and identically distributed (i.i.d.) sample. Second, we assume that
$\Exp[|X_{k, i}|] \leq C_k^{(\mu)}$ and $\Var[X_{k, i}] \leq C_k^{(\sigma^2)}$.
Lastly, we assume that $\pall = \sum_{k=1}^K p_k \geq C^{(\pall)}$.

\vspace{5mm}

\begin{algorithm}[htpb]
\caption{Pseudocode for \algname.}
\label{pseudocode}
\begin{algorithmic}[1]
\Function{\algname}{$\mathcal{D}$, $K$, $N_1$, $N_2$}
    \State $S_1, ..., S_K \gets \text{Stratify}(\mathcal{D}, K)$
    \For{each k in  [1, ..., K]}
        \State $R_k^{(1)} \gets \text{Sample}(S_k, N_1)$
        \State $X_k^{(1)} \gets \{f(x) \ | \ x \in R_k^{(1)}, \ O(x) = 1\}$
        \State $B_k^{(1)}$ $\gets$ $|X_k^{(1)}|$
        \State $\hat{\mu}_k$ $\gets$ $\sum_{i=1}^{B_k^{(1)}} X_{k, i}^{(1)} / B_k^{(1)}$ if $B_k^{(1)} > 0$ else $0$
        \State $\hat{p}_k$ $\gets$ $B_k^{(1)} / |R_k^{(1)}|$
        \State $\hat{\sigma}^2_k$ $\gets$ $\sum_{i=1}^{B_k^{(1)}} (X_{k, i}^{(1)} - \hat{\mu}_k)^2 / (B_k^{(1)}-1)$ if $B_k^{(1)} > 1$ else $0$
    \EndFor
    \For{each k in  [1, ..., K]}
        \State $\hat{T}_k$ $\gets$ $\sqrt{\hat{p}_k} \hat{\sigma}_k / \sum_{i=1}^K \sqrt{\hat{p}_i} \hat{\sigma}_i$
        \State $R_k^{(2)} \gets \text{Sample}(S_k, \ceil{N_2 \hat{T}_k})$
        \State $R_k^{(2)} \gets R_k^{(1)} + R_k^{(2)} \ \text{if} \ SampleReuse \ \text{else} \ R_k^{(2)}$
        \State $X_k^{(2)}$ $\gets$ $X_k^{(1)} + \{f(x) \ | \ x \not\in R_k^{(1)}, \ x \in R_k^{(2)}, \ O(x)=1\}$
        \State $B_k^{(2)} \gets |X_k^{(2)}|$
        \State $\hat{p}_k$ $\gets$ $B_k^{(2)} / |R_k^{(2)}| \ \text{if} \ SampleReuse \ \text{else} \ \hat{p}_k$
        \State $\hat{\mu}_k$ $\gets$ $\sum_{i=1}^{B_k^{(2)}} X_i^{(2)} / B_k^{(2)}$ if $B_k^{(2)} > 0$ else $0$
    \EndFor
    \State \Return $\sum_{k=1}^K \hat{p}_k \hat{\mu}_k / \sum_{k=1}^K \hat{p}_k$
\EndFunction
\end{algorithmic}
\end{algorithm}

For the purpose of this proof, $\text{SampleReuse}$ is $False$. However, we show that $\text{SampleReuse}$ set to $True$ does not change the error of \algname asymptotically via \Cref{corrollayx}. Furthermore, note that on line 13 we round up the number of allocated samples $N_2 \hat{T}_k$. This rounding ensures that our bounds hold; however, this means that we may use up to $N_2 + K$ samples. Hence, proof statements involving $N_2$ are only accurate up to $K$.

\section{Algorithm}
Refer to \Cref{pseudocode} for pseudocode of \algname and \Cref{nottable} for a description of notation used throughout the document. \algname creates the strata by splitting the dataset evenly between $K$ strata by ordering the records by proxy score.
 \algname will then perform a two-stage sampling procedure. \algname first
estimates $p_k$ and $\sigma_k$, which are the key quantities for determining
optimal allocation by sampling $N_1$ samples from each of the $K$ strata. Then,
 \algname will allocate the remaining samples proportional to our estimates of
the optimal allocation.  \algname will return the estimate $\sum_k \hat{p}_k
\hat{\mu}/\sum_k \hat{p}_k$ as the approximate answer.

\section{Analysis (Optimal Stratified Sampling Allocation)}
In this section, we assume that the true values of $p_k$ and $\sigma_k$ are known, and derive the optimal allocation according to knowledge of these quantities. We also assume that the sample selection process is deterministic and returns a number of draws equal to the expected number of draws which satisfy the predicate. We also derive an upper bound on the MSE based on this optimal allocation. To avoid fractional samples, we round up samples as appropriate.

\vspace{3mm}

Although the true $p_k$ and $\sigma_k$ are unknown in this setting, our algorithm is based on estimating them in Stage 1 and then allocating samples in Stage 2 based on those estimates. We will then later show in a future section that the estimates made in Stage 1 are tight enough with high probability to achieve matching asymptotic rates of error with respect to the optimal allocation.

\vspace{3mm}
\begin{prop}
Let $T_k$ represent the fraction of our sampling budget we want to allocate to stratum $S_k$. Suppose that $p_k$ and $\sigma_k$ are known, $\Exp[\hat{\mu}_k]=\mu_k$, and that we get $B_k = \ceil{p_kT_kN_2}$ samples for all $k$. Then, the choice of $T_k$ such that $T_k > 0$ and $\sum_{k=1}^K T_k=1$ which minimizes the loss function $\mathcal{L}(T_1, ..., T_K) = \sum_{k=1}^K  \frac{w_k^2\sigma_k^2}{p_kT_kN}$ of the unbiased estimator $\hmuall = \sum_{k=1}^K w_k \hat{\mu}_k$ is:
\begin{align}
    T_k^* = \frac{\sqrt{p_k}\sigma_k}{\sum_{i=1}^K \sqrt{p_i}\sigma_i}.
\end{align}
In addition, the choice $T_k=T_k^*$ will result in the following upper bound on the MSE:
\begin{align}
    E^* = \Exp\Big[(\hmuall - \muall)^2 \mid B_k = \ceil{p_k T_k^* N}\Big] \leq \frac{1}{N\pall^2} \cdot \left(\sum_{k=1}^K \sqrt{p_k}\sigma_k\right)^2.
\end{align}
\end{prop}
\begin{proof}

We can decompose the MSE into bias and variance. Note that the bias is zero due to the strata-specific estimates of the mean being unbiased.

\begin{align}
    \Exp[(\hmuall - \muall)^2 | B_k &= \ceil{p_kT_kN}] = \Exp[\hmuall - \muall]^2 + \Var[\hmuall] = \Var\left[\sum_{k=1}^K w_k\hat{\mu}_k\right]  \leq \mathcal{L}(T_1, ..., T_K) = \sum_{k=1}^K  \frac{w_k^2\sigma_k^2}{p_kT_kN}.
\end{align}

 We will use Lagrange multipliers to minimize the the loss function $\mathcal{L}$ with respect to $T_1$, ..., $T_K$ under the constraint $g$ which enforces that $\sum_{k=1}^K T_k = 1$.
\begin{align}
    \mathcal{L}(T_1, ..., T_K) = \sum_{k=1}^K \frac{w_k^2 \sigma_k^2}{p_k T_k N}, \ \ g(T_1, ..., T_K) = \sum_{k=1}^K T_k - 1.
\end{align}
\begin{align}
    \frac{\partial \mathcal{L}}{\partial T_k} + \lambda \frac{\partial g}{\partial T_k} &= 0\longrightarrow \frac{w_k^2 \sigma_k^2}{p_k T_k^2 N} = \lambda \\
    T_k &= \frac{w_k \sigma_k}{\sqrt{p_k \lambda N}} \\
    \sum_{k=1}^K T_k &= \sum_{k=1}^K  \frac{w_k \sigma_k}{\sqrt{p_k\lambda N}} = 1 \\
    \sum_{k=1}^K  \frac{w_k \sigma_k}{\sqrt{p_k N}} &= \sqrt{\lambda} \\
    \left(\sum_{i=k}^K  \frac{w_k\sigma_k}{\sqrt{p_k N}}\right)^2 &= \lambda.
\end{align}

Thus, the optimal choices of $T_1$, ..., $T_K$ to minimize $\mathcal{L}$ under the constraint $g$ is:
\begin{align}
    T_k^* &= \frac{w_k\sigma_k}{\sqrt{p_k\lambda N}} \\
    &= \frac{w_k \sigma_k}{\sqrt{p_kN} \left( \sum_{i=1}^K \frac{w_i \sigma_i}{\sqrt{p_i N}} \right)}  \\
    &= \frac{\sqrt{p_k}\sigma_k}{\sum_{i=1}^K \sqrt{p_i}\sigma_i}.
\end{align}
We will now derive an upper bound on the MSE given $T_k=T_k^*$. To deal with fractional samples, we round up $p_kT_k^*N$ or the number of draws each strata gets, which means that the bound derived below is still valid.
\begin{align}
    \Exp[(\hmuall - \muall)^2 | B_k = \ceil{p_k T_k^* N}] &\leq \sum_{k=1}^K  \frac{w_k^2 \sigma_k^2}{p_kT_k^*N} \\
    &= \sum_{k=1}^K  \frac{w_k^2 \sigma_k^2}{p_k\frac{\sqrt{p_k}\sigma_k}{\sum_{i=1}^K \sqrt{p_i}\sigma_i}N} \\
    &=  \frac{1}{\pall^2N} \cdot \sum_{k=1}^K \frac{p_k^2\sigma_k^2}{p_k\frac{\sqrt{p_k}\sigma_k}{\sum_{i=1}^K \sqrt{p_i}\sigma_i}} \\
    &=  \frac{1}{\pall^2N} \cdot \sum_{k=1}^K \frac{\sqrt{p_k}\sigma_k}{\frac{1}{\sum_{i=1}^K \sqrt{p_i}\sigma_i}} \\
    &=  \frac{1}{\pall^2 N}  \cdot \left(\sum_{k=1}^K \sqrt{p_k}\sigma_k\right)^2 \qedhere
\end{align}
\end{proof}

\pagebreak

\section{Analysis}
We will prove that with arbitrarily small failure probability $\delta$ over Stage 1, the expectation of the mean squared error of \algname over Stage 2 is
\begin{align}
    \Exp[(\hmuall - \muall)^2] &\leq O\left(N_1^{-1}\right) + O\left(N_2^{-1}\right) + O\left(N_1^{1/2}N_2^{-3/2}\right).
\end{align}

Our proof strategy is outlined below.
\begin{enumerate}
    \item Obtain high probability bounds over all the random variables in Stage 1. Namely, $p_k$, $w_k$, $\sigma_k$, $B_k^{(1)}$, and $T_k$. 
    \item Assuming the worst-case values described by these high probability bounds, we will upper bound the mean squared error of our algorithm in expectation over Stage 2.
\end{enumerate}

\subsection{Big-O Definition}
To clarify our usage of Big-O, we will formally define it as:
\begin{align}
T(...) = O(f(...)) \iff \exists \ C, \ N_0 \ \text{s.t.} \ \forall N \geq N_0 \text{,} \ 0 \leq T(...) \leq C \cdot f(...). 
\end{align}

We use Big-O notation with respect to $N_1$, $N_2$, $p_k$, and $\delta$. Hence, $C$ and $N_0$ must be independent of those quantities.

\subsection{Stage 1}
We will obtain high probability bounds on all the random variables in Stage 1. Namely, $p_k$, $w_k$, $\sigma_k$, $B_k^{(1)}$, and $T_k$. 

\begin{lemma}[$\hat{p}_k$ Upper Bound]\label{pupperbound} For small $\delta > 0$, the following holds for all $k$ simultaneously,
\begin{align}
     \sqrt{\hat{p}_k} \leq \sqrt{p_k + \sqrt{\frac{2\ln(1/\delta)p_k}{N_1}}} = \sqrt{p_k}\sqrt{ 1 + O\left( \sqrt{\frac{\ln(1/\delta)}{p_kN_1}} \right)}.
\end{align}
\end{lemma}
\begin{proof}
We apply \cite{chung2002connected} reproduced as \Cref{binomialbs} to upper bound $\hat{p}_k$. We set $Y_i \sim Bernoulli(p_k)$ and $a_i=\frac{1}{N_1}$ for all $i \in [N_1]$. Thus, $\nu=\sum_{i=1}^{N_1} a_i^2p_i = \frac{p_k}{N_1}$ and $Y=\sum_{i=1}^{N_1} a_iY_i=\hat{p}_k$. Then, according to the Lemma,

\begin{align}
    P\left( \hat{p}_k \geq p_k + \lambda \right) \leq exp\left(\frac{-\lambda^2}{2(\frac{p_k}{N_1} + \frac{\lambda}{3N_1})}\right).
\end{align}

We will set $\delta=exp\left(-\lambda^2/(2(\frac{p_k}{N_1} + \frac{\lambda}{3N_1}))\right)$ and $\lambda = \left(\sqrt{18N_1p_k \ln(1/\delta) + \ln(1/\delta)^2} - \ln(1/\delta)\right)/(3N_1)$. Thus, with probability at least $1-\delta$,
\begin{align}
     \hat{p}_k &\leq p_k + \frac{\sqrt{18N_1p_k \ln(1/\delta) + \ln(1/\delta)^2} - \ln(1/\delta)}{3 N_1} \\
     &\leq p_k + \frac{\sqrt{18N_1p_k \ln(1/\delta)}}{3 N_1} = p_k + \sqrt{\frac{2\ln(1/\delta)p_k }{N_1}}.
\end{align}

\vspace{3mm}

With probability at least $1-\delta$,
\begin{align}
    \sqrt{\hat{p}_k} &\leq \sqrt{p_k + \sqrt{\frac{2\ln(1/\delta)p_k}{N_1}}} \leq \sqrt{p_k} \cdot \sqrt{1 + \sqrt{\frac{2\ln(1/\delta)}{p_kN_1}}} = \sqrt{p_k}\sqrt{ 1 + O\left( \sqrt{\frac{\ln(1/\delta)}{p_kN_1}} \right)}.
\end{align}

We union bound across all strata $S_1, ..., S_k$. As a result, the bound above holds for all $k$ simultaneously with probability at least $1-K\delta$.
\end{proof} 

\begin{lemma}[$\hat{p}_k$ Lower Bound]\label{plowerbound} For small $\delta > 0$, the following holds for all $k$ simultaneously,
\begin{align}
     \sqrt{\hat{p}_k} \geq  \sqrt{p_k - \sqrt{\frac{2\ln(1/\delta)p_k}{N_1}}} = \sqrt{p_k} \sqrt{ 1 - O\left( \sqrt{\frac{\ln(1/\delta)}{p_kN_1}}\right)}.
\end{align}
\end{lemma}
\begin{proof}
We apply \cite{chung2002connected} reproduced as \Cref{binomialbs} to lower bound $\hat{p}_k$. We set $X_i \sim Bernoulli(p_k)$ and $a_i=\frac{1}{N_1}$ for all $i \in [N_1]$. Thus, $\nu=\sum_{i=1}^{N_1} a_i^2p_i = \frac{p_k}{N_1}$ and $X=\sum_{i=1}^{N_1} a_iX_i=\hat{p}_k$. Then, according to the Lemma:

\begin{align}
    P\left( \hat{p}_k < p_k - \lambda \right) \leq exp\left(\frac{-\lambda^2}{2(\frac{p}{N_1})}\right).
\end{align}

We will set $\delta=exp\left(\frac{-\lambda^2}{2(\frac{p}{N_1})}\right)$ and $\lambda = \frac{\sqrt{p}}{\sqrt{N_1}} \cdot \sqrt{2\ln(1/\delta)}$. Hence, with probability at least  $1-\delta$ (where $\delta$ is probability of failure):
\begin{align}
     \hat{p}_k \geq p_k - \frac{\sqrt{p_k}}{\sqrt{N_1}} \cdot \sqrt{2\ln(1/\delta)} = p_k - \sqrt{\frac{2\ln(1/\delta)p_k}{N_1}}.
\end{align}

In other words, with probability at least $1-\delta$,
\begin{align}
    \sqrt{\hat{p}_k} &\geq  \sqrt{p_k - \sqrt{\frac{2\ln(1/\delta)p_k}{N_1}}} \geq \sqrt{p_k} \sqrt{1 - \sqrt{\frac{2\ln(1/\delta)}{p_kN_1}}} \\
    &= \sqrt{p_k} \sqrt{ 1 - O\left( \sqrt{\frac{\ln(1/\delta)}{p_kN_1}}\right)}.
\end{align}
We then apply the union bound, so the above holds for all $k$ simultaneously with probability at least $1-K\delta$.
\end{proof}

\begin{lemma}[$\hat{w}_k$ Upper and Lower Bounds]\label{wbound}
For small $\delta > 0$, the bounds of \Cref{pupperbound} and \Cref{plowerbound} hold and

\begin{align}
    w_k - O\left( \frac{\ln(1/\delta)\sqrt{p_k}}{\sqrt{N_1}} \right) \leq \frac{p_k - \sqrt{\frac{2\ln(1/\delta)p_k}{N_1}}}{\pall + \sum_{k=1}^K \sqrt{\frac{2\ln(1/\delta)p_k}{N_1}}} \leq \hat{w_k} \leq \frac{p_k + \sqrt{\frac{2\ln(1/\delta)p_k}{N_1}}}{\pall - \sum_{k=1}^K \sqrt{\frac{2\ln(1/\delta)p_k}{N_1}}} \leq w_k + O\left( \frac{\ln(1/\delta)\sqrt{p_k}}{\sqrt{N_1}} \right).
\end{align}
\end{lemma}
\begin{proof}
Recall that $\hat{w}_k=\hat{p}_k/\hpall$, where $\hpall=\sum_{k=1}^K \hat{p}_k$. (29) holds on the event that the bounds of \Cref{pupperbound} and \Cref{plowerbound} hold for all $k$ simultaneously, which occurs with probability at least $1-2K\delta$ according to the union bound. For the remainder of the proof of this Lemma, we condition on this event and prove (29). Specifically, the condition is that for each $k$,

\begin{align}
    p_k - \sqrt{\frac{2\ln(1/\delta)p_k}{N_1}} \leq \hat{p}_k \leq p_k + \sqrt{\frac{2\ln(1/\delta)p_k}{N_1}}.
\end{align}

Hence, by summing these inequalities across $k$, we have that with probability at least $1-2K\delta$,

\begin{align}
    \pall - \sum_{k=1}^K \sqrt{\frac{2\ln(1/\delta)p_k}{N_1}}\leq \hpall \leq \pall + \sum_{k=1}^K \sqrt{\frac{2\ln(1/\delta)p_k}{N_1}}.
\end{align}

We will now maximize or minimize the numerator and denominator appropriately in $\hat{w}_k=\hat{p}_k/\hpall$.

\begin{align}
    \frac{p_k - \sqrt{\frac{2\ln(1/\delta)p_k}{N_1}}}{\pall + \sum_{k=1}^K \sqrt{\frac{2\ln(1/\delta)p_k}{N_1}}} \leq \hat{w_k} \leq \frac{p_k + \sqrt{\frac{2\ln(1/\delta)p_k}{N_1}}}{\pall - \sum_{k=1}^K \sqrt{\frac{2\ln(1/\delta)p_k}{N_1}}}.
\end{align}

We will now simplify this bound with Big O. Note that $\pall \geq C^{(\pall)}$.

\begin{align}
    \hat{w_k} &\leq \frac{p_k + \sqrt{\frac{2\ln(1/\delta)p_k}{N_1}}}{\pall - \sum_{k=1}^K \sqrt{\frac{2\ln(1/\delta)p_k}{N_1}}} \\
    &= \frac{\frac{1}{\pall}}{\frac{1}{\pall}} \cdot \frac{p_k}{\pall - \sum_{k=1}^K \sqrt{\frac{2\ln(1/\delta)p_k}{N_1}}} + \frac{\frac{1}{\pall}}{\frac{1}{\pall}} \cdot \frac{ \sqrt{\frac{2\ln(1/\delta)p_k}{N_1}}}{\pall - \sum_{k=1}^K \sqrt{\frac{2\ln(1/\delta)p_k}{N_1}}} \\
    &= \frac{w_k}{1 - \sum_{k=1}^K \sqrt{\frac{2\ln(1/\delta)p_k}{\pall^2N_1}}}  + \frac{ \sqrt{\frac{2\ln(1/\delta)p_k}{\pall^2N_1}}}{1 - \sum_{k=1}^K \sqrt{\frac{2\ln(1/\delta)p_k}{\pall^2N_1}}} \\ 
    &\leq \left(w_k + \sqrt{\frac{2\ln(1/\delta)p_k}{\pall^2N_1}} \right) \cdot \left (1 + \sum_{k=1}^K \sqrt{\frac{2\ln(1/\delta)p_k}{\pall^2N_1}} \right).
\end{align}
\begin{align}
    &\leq \left(w_k + \sqrt{\frac{2\ln(1/\delta)p_k}{\pall^2N_1}} \right) \cdot \left (1 + O\left(\sqrt{\frac{\ln(1/\delta)}{N_1}}\right) \right) \\
    &\leq w_k + \sqrt{\frac{2\ln(1/\delta)p_k}{\pall^2N_1}} + O\left(  \frac{p_k\sqrt{\ln(1/\delta)}}{\sqrt{N_1}} \right) + O\left(\frac{\ln(1/\delta)\sqrt{p_k}}{N_1} \right) \\
    &\leq w_k + O\left( \frac{\sqrt{\ln(1/\delta)p_k}}{\sqrt{N_1}} \right) + O\left(  \frac{p_k\sqrt{\ln(1/\delta)}}{\sqrt{N_1}} \right) + O\left( \frac{\ln(1/\delta)\sqrt{p_k}}{N_1} \right) \\
    &\leq w_k + O\left(\frac{\ln(1/\delta)\sqrt{p_k}}{\sqrt{N_1}} \right).
\end{align}

We repeat these steps for the lower bound. Thus,

\begin{align}
    w_k - O\left( \frac{\ln(1/\delta)\sqrt{p_k}}{\sqrt{N_1}} \right) \leq \hat{w}_k \leq w_k + O\left( \frac{\ln(1/\delta)\sqrt{p_k}}{\sqrt{N_1}} \right).
\end{align}
\end{proof}

\begin{lemma}[$B_k^{(1)}$ Lower Bound]\label{blowerbound} For small $\delta > 0$, the following is true for all $k$,
\begin{align}
    B_k^{(1)} \geq p_kN_1 - \sqrt{2\ln(1/\delta)p_kN_1}.
\end{align}
\end{lemma}
\begin{proof}
We apply \cite{princetonchernoff} reproduced as \Cref{chernoffbs} to lower bound $B_k^{(1)}$. Note that $B_k^{(1)} \sim Binomial(N_1, p_k)$ and that equivalently $B_k^{(1)} = \sum_{i=1}^{N_1} Y_i$ where $Y_i \sim Bernoulli(p_k)$. Hence,

\begin{align}
    P\left( B_k^{(1)} \leq \left(1 - \epsilon\right)\left( p_kN_1 \right) \right) \leq exp\left( \frac{-\epsilon^2p_kN_1}{2}\right).
\end{align}

We set $\delta=exp\left( \frac{-\epsilon^2p_kN_1}{2}\right)$ and $\epsilon=\sqrt{\frac{2\ln(1/\delta)}{p_kN_1}}$. Thus, with probability at least $1-\delta$:

\begin{align}
    B_k^{(1)} \geq p_kN_1 - \sqrt{2\ln(1/\delta)p_kN_1}.
\end{align}

We then union bound across all $k$. As a result, with probability at least $1-K\delta$, (44) holds for all $k$ simultaneously.
\end{proof}

\begin{lemma}[$\sigma_k$ Upper and Lower Bound for $B_k^{(1)} \geq 2$]\label{sigmabound} For small $\delta > 0$, the following holds for all $k$ such that $B_k^{(1)} \geq 2$.
\begin{align}
    \sigma_k \cdot \sqrt{1 - O\left(\sqrt{\frac{\ln(1/\delta)}{p_kN_1}}\right)} &\leq \hat{\sigma_k}   \leq \sigma_k \cdot \sqrt{1 + O\left(\sqrt{\frac{\ln(1/\delta)}{p_kN_1}}\right)}.
\end{align}
\end{lemma}
\begin{proof}
To obtain a concentration inequality on $\hat{\sigma}_k$, we will apply bounded differences on the unbiased sample variance estimator for $B_k^{(1)} > 2$, which is:
\begin{align}
    \hat{\sigma}_k^2 = \frac{1}{B_k^{(1)}-1} \sum_{i=1}^{B_k^{(1)}} (X_{k, i} - \hat{\mu}_k)^2.
\end{align}

Note that the unbiased sampled variance is a U-Statistic $U_n$ that arises from taking $g(X_{k, i}, X_{k, j})=\frac{1}{2}(X_{k, i} - X_{k, j})^2$ (\cite{ustatistictovariance} reproduced as \Cref{bdtosv}). For bounded U-Statistics, we can apply bounded differences to obtain a concentration bound (\cite{ustatistic} reproduced as \Cref{bdinq}). Namely, where $g(X_{k, i}, X_{k, j}) \leq b$,
\begin{align}
    P(|U_n - \Exp[U_n]| \geq t) \leq 2\exp(-\frac{nt^2}{8b^2}).
\end{align}
We use this concentration inequality with $n=B_k^{(1)}$ and $g(X_{k, i}, X_{k, j}) \leq b = C_k^{(\mu^2)}$. Hence,
\begin{align}
    P(|\hat{\sigma_k}^2 - \sigma_k^2| \geq t) \leq 2\exp(-\frac{B_k^{(1)}t^2}{8C_k^{(\mu^4)}}).
\end{align}
We then set $2\delta=2\exp(-\frac{B_k^{(1)}t^2}{8C_k^{(\mu^4)}})$ and $t=\sqrt{8\ln(1/\delta)C_k^{(\mu^4)}/B_k^{(1)}}$. Thus, with probability at least $1 - 2\delta$,
\begin{align}
    \sigma_k^2 - \sqrt{\frac{8\ln(1/\delta)C_k^{(\mu^4)}}{B_k^{(1)}}} &\leq \hat{\sigma_k}^2   \leq \sigma_k^2 + \sqrt{\frac{8\ln(1/\delta)C_k^{(\mu^4)}}{B_k^{(1)}}}.
\end{align}
Using \Cref{blowerbound}, we have $B_k^{(1)} \geq p_kN_1 - \sqrt{2\ln(1/\delta)p_kN_1}$ with probability at least $1-K\delta$. Note that if that bound holds, then $p_k \geq p_* = \frac{2\ln(1/\delta)+2\sqrt{\ln(1/\delta)}+2}{N_1}$, which allows us to simplify $(53)$ into $(54)$.
\begin{align}
    \sigma_k^2 - \sqrt{\frac{8\ln(1/\delta)C_k^{(\mu^4)}}{p_kN_1 - \sqrt{2\ln(1/\delta)p_kN_1}}} &\leq \hat{\sigma_k}^2   \leq \sigma_k^2 + \sqrt{\frac{8\ln(1/\delta)C_k^{(\mu^4)}}{p_kN_1 - \sqrt{2\ln(1/\delta)p_kN_1}}} \\
    \sigma_k^2 - \sqrt{\frac{8\ln(1/\delta)C_k^{(\mu^4)}}{p_kN_1 \left(1 - \sqrt{\frac{2\ln(1/\delta)}{p_kN_1}}\right)}}&\leq \hat{\sigma_k}^2   \leq \sigma_k^2 + \sqrt{\frac{8\ln(1/\delta)C_k^{(\mu^4)}}{p_kN_1 \left(1 - \sqrt{\frac{2\ln(1/\delta)}{p_kN_1}}\right)}} \\
    \sigma_k^2 - \sqrt{\frac{8\ln(1/\delta)C_k^{(\mu^4)}}{p_kN_1} \cdot \left(1 + \sqrt{\frac{2\ln(1/\delta)}{p_kN_1}}\right)} &\leq \hat{\sigma_k}^2   \leq \sigma_k^2 + \sqrt{\frac{8\ln(1/\delta)C_k^{(\mu^4)}}{p_kN_1} \cdot \left(1 + \sqrt{\frac{2\ln(1/\delta)}{p_kN_1}}\right)} \\
    \sigma_k^2 - \sqrt{\frac{8\ln(1/\delta)C_k^{(\mu^4)}}{p_kN_1} + \frac{8\sqrt{2}\ln(1/\delta)^{3/2}}{p_k^{3/2}N_1^{3/2}}} &\leq \hat{\sigma_k}^2   \leq \sigma_k^2 + \sqrt{\frac{8\ln(1/\delta)C_k^{(\mu^4)}}{p_kN_1} + \frac{8\sqrt{2}\ln(1/\delta)^{3/2}}{p_k^{3/2}N_1^{3/2}}} \\
    \sigma_k^2 - \sqrt{O\left(\frac{\ln(1/\delta)}{p_kN_1}\right)} &\leq \hat{\sigma_k}^2   \leq \sigma_k^2 + \sqrt{O\left(\frac{\ln(1/\delta)}{p_kN_1}\right)} \\
    \sqrt{\sigma_k^2 - \sqrt{O\left(\frac{\ln(1/\delta)}{p_kN_1}\right)}} &\leq \hat{\sigma_k}   \leq \sqrt{\sigma_k^2 + \sqrt{O\left(\frac{\ln(1/\delta)}{p_kN_1}\right)}} \\
    \sigma_k \cdot \sqrt{1 - O\left(\sqrt{\frac{\ln(1/\delta)}{p_kN_1}}\right)}&\leq \hat{\sigma_k}   \leq \sigma_k \cdot \sqrt{1 + O\left(\sqrt{\frac{\ln(1/\delta)}{p_kN_1}}\right)}.
\end{align}
We union bound across the strata and intersection bound our application of Lemma 4, and get that the above holds with probability at least $(1-2K\delta)(1 - K\delta)=1-3K\delta+2K^2\delta^2 \geq 1 - 3K\delta$. The derivation above relies on $\sigma_k \neq 0$ because we divide by $\sigma_k$ in (55). However, if $\sigma_k = 0$, then $\hat{\sigma}_k = \sigma_k$, so the bound still holds. 
\end{proof}

\begin{lemma}[$\sqrt{p_k}\sigma_k$ Upper Bound for $B_k^{(1)} < 2$]\label{psigmabound}
For $B_k^{(1)} < 2$, the following holds for all $k$:
\begin{align}
    \sqrt{\hat{p}_k}\hat{\sigma}_k &\leq \sqrt{p_k}\sigma_k + O\left(\frac{1}{\sqrt{N_1}}\right).
\end{align}
\end{lemma}
\begin{proof}
In the cases where $B_k^{(1)} < 2$, we will upper bound $\sqrt{\hat{p}_k}\hat{\sigma}_k$ directly instead of combining our concentration inequalities for $\sqrt{\hat{p}_k}$ and $\hat{\sigma}_k$ separately. 
\begin{align}
    \sqrt{\hat{p}_k}\hat{\sigma}_k &\leq  \sqrt{\frac{2}{N_1}}\hat{\sigma}_k \leq \sqrt{p_k}\sigma_k + \sqrt{\frac{2}{N_1}}C_k^{(\sigma)} \leq \sqrt{p_k}\sigma_k + O\left(\frac{1}{\sqrt{N_1}}\right).
\end{align}
\end{proof}

\begin{lemma}[$T_k$ Lower Bound]\label{tlowerbound}
For small $\delta > 0$, the following is true simultaneously for all $k$ where $p_k > p_* = \frac{2\ln(1/\delta)+2\sqrt{\ln(1/\delta)}+2}{N_1}$,
\begin{align}
    \hat{T}_k \geq T_k \left(1 - O\left( \sqrt{\frac{\ln(1/\delta)}{p_kN_1}}\right)\right).
\end{align}
\end{lemma}
\begin{proof}
 Our goal is to lower bound $\hat{T}_k= \frac{\sqrt{\hat{p}_k}\hat{\sigma}_k}{\sum_{i=1}^K \sqrt{\hat{p}_i}\hat{\sigma}_i}$. To accomplish this, we will upper bound the denominator and lower bound the numerator. To lower bound the numerator, we will apply our concentration inequalities for $\sqrt{\hat{p}_k}$ and $\hat{\sigma}_k$ from \Cref{plowerbound} and \Cref{sigmabound}. We only lower bound for $k$ where $B_k^{(1)} \geq 2$. Through \Cref{blowerbound}, we know that for $B_k^{(1)} \geq 2$ to be true we can condition on $p_k > p_* =  \frac{2\ln(1/\delta)+2\sqrt{\ln(1/\delta)}+2}{N_1}$. This allows us to satisfy the conditions of \Cref{sigmabound}.
\begin{align}
    \sqrt{\hat{p}}_k\sigma_k &\geq \sqrt{p_k}\sigma_k \sqrt{1-O\left( \sqrt{\frac{\ln(1/\delta)}{p_kN_1}}\right)} \sqrt{1-O\left( \sqrt{\frac{\ln(1/\delta)}{p_kN_1}}\right)} \\
    &\geq  \sqrt{p_k}\sigma_k \left(1-O\left( \sqrt{\frac{\ln(1/\delta)}{p_kN_1}}\right)\right) \geq  \sqrt{p_k}\sigma_k - O\left( \sqrt{\frac{\ln(1/\delta)}{N_1}}\right).
\end{align}
To maximize the denominator, we also apply \Cref{pupperbound} and split the cases where $p_k > p_*$ and where $p_k \leq p_*$.
\begin{align}
    \sum_{i=1}^K \sqrt{\hat{p}_i}\hat{\sigma}_i &\leq \sum_{p_k > p_*}^K \sqrt{p_i}\sigma_i \sqrt{1+O\left( \sqrt{\frac{\ln(1/\delta)}{p_iN_1}}\right)} \sqrt{1+O\left( \sqrt{\frac{\ln(1/\delta)}{p_iN_1}}\right)} + \sum_{p_k > p_*}^K \sqrt{p_i}\sigma_i + O\left(\frac{1}{\sqrt{N_1}}\right) \\
    &\leq \sum_{p_k < p_*}^K \sqrt{p_i}\sigma_i \left(1+O\left( \sqrt{\frac{\ln(1/\delta)}{p_iN_1}}\right)\right) + \sum_{p_k \leq p_*}^K \sqrt{p_i}\sigma_i + O\left(\frac{1}{\sqrt{N_1}}\right)\\
    &\leq \sum_{i=1}^K \sqrt{p_i}\sigma_i + O\left( \sqrt{\frac{\ln(1/\delta)}{N_1}}\right).
\end{align}

Thus, for $p_k > p_* = \frac{2\ln(1/\delta)+2\sqrt{\ln(1/\delta)}+2}{N_1}$,
\begin{align}
    \hat{T}_k= \frac{\sqrt{\hat{p}_k}\hat{\sigma}_k}{\sum_{i=1}^K \sqrt{\hat{p}_i}\hat{\sigma}_i} &\geq \frac{\sqrt{p_k}\sigma_k - O\left( \sqrt{\frac{\ln(1/\delta)}{N_1}}\right)}{\sum_{i=1}^K \sqrt{p_i}\sigma_i + O\left( \sqrt{\frac{\ln(1/\delta)}{N_1}}\right)} = \frac{\frac{1}{\sum_{i=1}^K \sqrt{p_i}\sigma_i}}{\frac{1}{\sum_{i=1}^K \sqrt{p_i}\sigma_i}} \cdot \frac{\sqrt{p_k}\sigma_k - O\left( \sqrt{\frac{\ln(1/\delta)}{N_1}}\right)}{\sum_{i=1}^K \sqrt{p_i}\sigma_i + O\left( \sqrt{\frac{\ln(1/\delta)}{N_1}}\right)} \\
    &\geq \frac{T_k - O\left( \sqrt{\frac{\ln(1/\delta)}{N_1}}\right)}{1 + O\left( \sqrt{\frac{\ln(1/\delta)}{N_1}}\right)} \geq \left(T_k - O\left( \sqrt{\frac{\ln(1/\delta)}{N_1}}\right)\right)\left(1 - O\left( \sqrt{\frac{\ln(1/\delta)}{N_1}}\right)\right) \\
    &\geq T_k - O\left( \sqrt{\frac{\ln(1/\delta)}{N_1}}\right) - O\left(\sqrt{\frac{p_k\ln(1/\delta)}{N_1}}\right) + O\left(\frac{\ln(1/\delta)}{N_1}\right) \\
    &\geq T_k - O\left( \sqrt{\frac{\ln(1/\delta)}{N_1}}\right) \geq T_k \left(1 - O\left( \sqrt{\frac{\ln(1/\delta)}{p_kN_1}}\right)\right) \qedhere
\end{align}
Through the union bound on \Cref{pupperbound}, \Cref{plowerbound}, \Cref{sigmabound}, (73) occurs with probability at least $1-5K\delta$. 
\end{proof}

\subsection{Stage 2}
Recall that we use high probability bounds over Stage 1 and take the expectation over Stage 2. In this section, we derive bounds on quantities in Stage 2 that will help us upper bound the mean squared error of our algorithm. 

\begin{lemma}[$B_k^{(2)}$ Lower Bound]\label{b2lowerbound} If \Cref{tlowerbound} holds, the following is true for all $k$ where $p_k > p_* = \frac{2\ln(1/\delta)+2\sqrt{\ln(1/\delta)}+2}{N_1}$ with probability at least $1-\gamma$,

\begin{align}
    B_k^{(2)} \geq p_kT_kN_2 \left(1 - O\left( \sqrt{\frac{1}{p_kN_1}}\right) - O\left( \sqrt{\frac{\ln(1/\gamma)}{N_2p_k^{3/2}}}\right)\right).
\end{align}
\end{lemma}
\begin{proof}
Recall that $B_k^{(2)} \sim Binomial(N_2\hat{T}_k, p_k)$. Thus, using \cite{princetonchernoff} reproduced as \Cref{chernoffbs}, we have that:

\begin{align}
    P\left( B_k^{(2)} \leq \left(1 - \epsilon\right)\left( p_k \hat{T}_k N_2  \right) \right) \leq exp\left( \frac{-\epsilon^2 p_k \hat{T}_k N_2  }{2}\right).
\end{align}

We set $\gamma=exp\left( \frac{-\epsilon^2p_k\hat{T}_kN_2}{2}\right)$ and $\epsilon=\sqrt{\frac{2\ln(1/\gamma)}{p_k\hat{T}_kN_2}}$. Thus,we have that with probability $1-\gamma$ and if \Cref{tlowerbound} holds:

\begin{align}
    B_k^{(2)} &\geq p_k\hat{T}_kN_2\left(1 - \sqrt{\frac{2\ln(1/\gamma)}{p_k\hat{T}_kN_2}}\right) \geq p_k\hat{T}_kN_2 - \sqrt{2\ln(1/\gamma)p_k\hat{T}_kN_2} \geq p_k\hat{T}_kN_2\left(1 - \sqrt{\frac{2\ln(1/\gamma)}{p_k\hat{T}_kN_2}}\right).
\end{align}

We lower bound the number of draws $\hat{T}_kN_2$ by conditioning on \Cref{tlowerbound} which says that $\hat{T}_k \geq T_k \left(1 - O\left( \sqrt{\frac{\ln(1/\delta)}{p_kN_1}}\right)\right)$.  Thus, with probability at least $1-\gamma$,

\begin{align}
    &\geq p_kT_kN_2\left(1 - O\left( \sqrt{\frac{\ln(1/\delta)}{p_kN_1}}\right)\right)\left(1 - \sqrt{\frac{2\ln(1/\gamma)}{p_kT_kN_2\left(1 - O\left( \sqrt{\frac{\ln(1/\delta)}{p_kN_1}}\right)\right)}}\right) \\
    &\geq p_kT_kN_2 \left(1 - O\left( \sqrt{\frac{\ln(1/\delta)}{p_kN_1}}\right) - \sqrt{\frac{2\ln(1/\gamma)}{p_kT_kN_2\left(1 - O\left(\frac{\ln(1/\delta)}{\sqrt{p_kN_1}}\right)\right)}}\right) \\ 
    &\geq p_kT_kN_2 \left(1 - O\left( \sqrt{\frac{\ln(1/\delta)}{p_kN_1}}\right) - \sqrt{\frac{2\ln(1/\gamma)}{p_kT_kN_2} \cdot \left(1 + O\left(\frac{\ln(1/\delta)}{\sqrt{p_kN_1}}\right)\right)}  \right) \\ 
    &\geq p_kT_kN_2 \left(1 - O\left( \sqrt{\frac{\ln(1/\delta)}{p_kN_1}}\right) - \sqrt{\frac{2\ln(1/\gamma)}{p_kT_kN_2} + O\left(\frac{\ln(1/\delta)\ln(1/\gamma)}{N_2p_k^{2}\sqrt{N_1}}\right)}  \right) \\ 
    &\geq p_kT_kN_2 \left(1 - O\left( \sqrt{\frac{1}{p_kN_1}}\right) - O\left( \sqrt{\frac{\ln(1/\gamma)}{N_2p_k^{3/2}}}\right)\right) \qedhere
\end{align}
We use the fact that $p_k > p_* \geq 1/N_1$ on $(74)$.
\end{proof}

\begin{lemma}[$w_k^2P(B_k^{(2)}=0)$ Upper Bound]\label{mubound}
If \Cref{tlowerbound} holds for $p_k > p_*$ and if $N_2 = \omega(N_1^{3/4})$,
\begin{align}
    p_k^x P(B_k^{(2)}=0) &\leq O\left(\frac{1}{N_1^{x}}\right) + O\left(\frac{\sqrt{N_1}}{N_2^2}\right).
\end{align}
\end{lemma}
\begin{proof}
\begin{align}
    p_k^x P(B_k^{(2)}=0) &= p_k^x (1-p_k)^{\ceil{\hat{T}_kN_2}} \leq p_k^x (1-p_k)^{\hat{T}_kN_2}.
\end{align}
If $p_k \leq p_* = \frac{2\ln(1/\delta)+2\sqrt{\ln(1/\delta)}+2}{N_1}$,
\begin{align}
    p_k^x (1-p_k)^{\hat{T}_kN_2} \leq O\left(\frac{1}{N_1^{x}}\right).
\end{align}
Else if $p_k > p_*$, which satisfies the conditions of \Cref{tlowerbound},
\begin{align}
    p_k^x (1-p_k)^{\hat{T}_kN_2} &\leq (1 - p_k)^{T_kN_2\left(1 - O\left(\frac{1}{\sqrt{p_kN_1}}\right)\right)} \\
    &\leq (1 - p_k)^{N_2\left(O\left(\sqrt{p_k}\right) - O\left(\frac{1}{\sqrt{N_1}}\right)\right)} \\
    &\leq (1 - p_k)^{O\left(\frac{N_2}{N_1^{1/4}}\right)} .
\end{align}
We assume that $N_2 = \omega(N_1^{3/4})$ which means that (80) is decreasing exponentially. Hence, we can say that $(1 - p_k)^{O\left(\frac{N_2}{N_1^{1/4}}\right)} = O\left(\frac{\sqrt{N_1}}{N_2^2}\right)$.
\begin{align}
    p_k^x P(B_k^{(2)}=0) &\leq O\left(\frac{1}{N_1^{x}}\right) + O\left(\frac{\sqrt{N_1}}{N_2^2}\right) \qedhere
\end{align}
\end{proof}

\subsection{Final Result}
\begin{theorem}\label{finaltheorem}
For small $\delta > 0$ over Stage 1 (\Cref{pupperbound}, \Cref{plowerbound}, \Cref{wbound}, \Cref{blowerbound}, \Cref{sigmabound}, \Cref{psigmabound}, and \Cref{tlowerbound}) and if $N_2 = \omega(N_1^{3/4})$, the expectation of the mean squared error of \algname over Stage 2 is
\begin{align}
    \Exp[(\hmuall - \muall)^2] &\leq O\left(N_1^{-1}\right) + O\left(N_2^{-1}\right) + O\left(N_1^{1/2}N_2^{-3/2}\right).
\end{align}
\end{theorem}
\begin{proof}
By \Cref{pupperbound}, \Cref{plowerbound}, \Cref{wbound}, \Cref{blowerbound}, \Cref{sigmabound}, \Cref{psigmabound}, and \Cref{tlowerbound}, we have that for small $\delta > 0$:
    \begin{align}
        \sqrt{p_k}\sqrt{ 1 - O\left( \sqrt{\frac{\ln(1/\delta)}{p_kN_1}} \right)} \leq \sqrt{\hat{p}_k} \leq \sqrt{p_k}\sqrt{ 1 + O\left( \sqrt{\frac{\ln(1/\delta)}{p_kN_1}} \right)}.
    \end{align}

    \begin{align}
    w_k - O\left( \frac{\ln(1/\delta)\sqrt{p_k}}{\sqrt{N_1}} \right) \leq \hat{w_k} \leq w_k + O\left( \frac{\ln(1/\delta)\sqrt{p_k}}{\sqrt{N_1}} \right)
    \end{align}

    \begin{align}
    B_k^{(1)} \geq p_kN_1 - \sqrt{2\ln(1/\delta)p_kN_1}.
\end{align}

\begin{align}
    \sigma_k \cdot \sqrt{1 - O\left(\sqrt{\frac{\ln(1/\delta)}{p_kN_1}}\right)} &\leq \hat{\sigma_k}   \leq \sigma_k \cdot \sqrt{1 + O\left(\sqrt{\frac{\ln(1/\delta)}{p_kN_1}}\right)}.
\end{align}

\begin{align}
    \hat{T}_k \geq T_k \left(1 - O\left( \sqrt{\frac{\ln(1/\delta)}{p_kN_1}}\right)\right).
\end{align}

Let $\mathcal{E}$ be the event that all these inequalities are satisfied. For the remainder of the proof, we condition on $\mathcal{E}$ and take the expectation over the randomness in Stage 2. We now begin to decompose the mean squared error.
\begin{align}
    \Exp[(\hmuall - \muall)^2] &= \Exp\left[(\sum_{k=1}^K \hat{w}_k\hat{\mu}_k - w_k\mu_k)^2 \right].
\end{align}
We know that $\Exp[\hat{\mu}_k] = P(B_k^{(2)} > 0) \cdot \mu_k$. As a result, $\Exp[\hat{\mu}_k] + \mu_kP(B_k^{(2)}=0) = \mu_k$. Hence,
\begin{align}
    = \Exp\left[\left( \sum_{k=1}^K \hat{w}_k(\hat{\mu}_k - \Exp[\hat{\mu}_k]) + (\hat{w}_k - w_k)\Exp[\hat{\mu_k}] - w_k\mu_k P(B_k=0) \right)^2 \right]
\end{align}
We now use that $\Exp[(A+B)^2] = \Exp[A^2 + 2AB + B^2]$, so if $\Exp[A]=0$ then $\Exp[(A+B)^2] = \Exp[A^2] + \Exp[B^2]$. We set $A=\hat{w}_k(\hat{\mu}_k - \Exp[\hat{\mu}_k])$ and $B=(\hat{w}_k - w_k)\Exp[\hat{\mu_k}] - w_k\mu_k P(B_k=0)$. Note that $\Exp[A] = 0$. Thus,
\begin{align}
    &= \Exp\left[(\sum_{k=1}^K \hat{w}_k(\hat{\mu}_k - \Exp[\hat{\mu}_k]) )^2\right] + \Exp\left[(\sum_{k=1}^K (\hat{w}_k - w_k)\Exp[\hat{\mu_k}] - w_k\mu_k P(B_k=0) )^2\right] \\
    &\leq \sum_{k=1}^K \hat{w}_k^2\Var(\hat{\mu}_k) + \left[\max_{k} C_k^{\mu_k^2}\right] \Exp\left[\left(\sum_{k=1}^K w_k P(B_k^{(2)} = 0)  + (\hat{w}_k - w_k)\right)^2\right].
\end{align}
We will now separately upper bound the two terms in this expression. For the first term, notice that
\begin{align}
    \Var[\hat{\mu}_k] &= \Exp[\hat{\mu}_k^2] - \Exp[\hat{\mu}_k]^2 \\
    &= P(B_k^{(2)} > 0)(\Exp[\frac{\sigma_k^2}{B_k^{(2)}} | B_k^{(2)} > 0] + \mu_k^2) - P(B_k^{(2)} > 0)^2\mu_k^2 \\
    &= P(B_k^{(2)} > 0)\Exp[\frac{\sigma_k^2}{B_k^{(2)}} | B_k^{(2)} > 0] + P(B_k^{(2)} > 0)(1 - P(B_k^{(2)} > 0))\mu_k^2 \\
    &= P(B_k^{(2)} > 0)\Exp[\frac{\sigma_k^2}{B_k^{(2)}} | B_k^{(2)} > 0] + P(B_k^{(2)} > 0)P(B_k^{(2)} = 0)\mu_k^2 \\
    &\leq \Exp[\frac{\sigma_k^2}{B_k^{(2)}} | B_k^{(2)} > 0] + P(B_k^{(2)} = 0)\mu_k^2 \\
    \sum_{k=1}^K \hat{w}_k^2 \Var[\hat{\mu}_k] &\leq \sum_{k=1}^K  \hat{w}_k^2 \left( \Exp[\frac{\sigma_k^2}{B_k^{(2)}} | B_k^{(2)} > 0] +  P(B_k^{(2)} = 0)\mu_k^2 \right) \\
    &\leq \sum_{k=1}^K  \left(O\left(p_k^2\right) + O\left(\frac{p_k^{3/2}}{\sqrt{N_1}}\right)\right) \left( \Exp[\frac{\sigma_k^2}{B_k^{(2)}} | B_k^{(2)} > 0] + P(B_k^{(2)} = 0)\mu_k^2 \right) \\
    &\leq \sum_{k=1}^K  \left(O\left(p_k^2\right) + O\left(\frac{p_k^{3/2}}{\sqrt{N_1}}\right)\right) \left( \Exp[\frac{\sigma_k^2}{B_k^{(2)}} | B_k^{(2)} > 0] + P(B_k^{(2)} = 0)\mu_k^2 \right)
\end{align}

We split the cases where $p_k$ is small and $p_k$ is large via $p_*=\frac{2\ln(1/\delta)+2\sqrt{\ln(1/\delta)}+2}{N_1}$. This value of $p_*$ is chosen as according to the conditions of \Cref{tlowerbound}.
\begin{align}
    &\leq \sum_{p_k > p_*}^K  \left(O\left(p_k^2\right) + O\left(\frac{p_k^{3/2}}{\sqrt{N_1}}\right)\right) \left( \Exp[\frac{\sigma_k^2}{B_k^{(2)}} | B_k^{(2)} > 0] + P(B_k^{(2)} = 0)\mu_k^2 \right) + \sum_{p_k \leq p_*}^K O\left(\frac{1}{N_1^2}\right)
\end{align}
Note that $p_k^2P(B_k^{(2)}=0) \leq O\left(\frac{1}{N_1}\right) + O\left(\frac{\sqrt{N_1}}{N_2^2}\right)$ according to \Cref{mubound}.
\begin{align}
    &\leq \sum_{p_k > p_*}^K  \left(O\left(p_k^2\right) + O\left(\frac{p_k^{3/2}}{\sqrt{N_1}}\right)\right) \Exp[\frac{\sigma_k^2}{B_k^{(2)}} | B_k^{(2)} > 0] + O\left(\frac{1}{N_1}\right) + O\left(\frac{\sqrt{N_1}}{N_2^2}\right)
\end{align}

We will now apply \Cref{b2lowerbound} to upper bound $\Exp[\frac{\sigma_k^2}{B_k^{(2)}} | B_k^{(2)} > 0]$. Namely, we will split the expectation into two cases, one where $B_k^{(2)} \geq F_k$ and another where $B_k < F_k$. Specifically, we apply it such that $B_k^{(2)} \geq F_k = p_kT_kN_2 \left(1 - O\left( \sqrt{\frac{1}{p_kN_1}}\right) - O\left( \sqrt{\frac{\ln(1/\gamma)}{N_2p_k^{3/2}}}\right)\right)$ with failure probability $\gamma=\frac{1}{e^{\sqrt{N_1}}} \leq O\left(\frac{1}{N_1}\right)$.

{\small
\begin{align}
    &\leq \sum_{p_k > p_*}^K  \left(O\left(p_k^2\right) + O\left(\frac{p_k^{3/2}}{\sqrt{N_1}}\right)\right)  \left( \Exp[\frac{\sigma_k^2}{B_k^{(2)}} | B_k^{(2)} > 0]\right) +  O\left(\frac{1}{N_1}\right) + O\left(\frac{\sqrt{N_1}}{N_2^2}\right)\\
    &\leq \sum_{p_k > p_*}^K  \left(O\left(p_k^2\right) + O\left(\frac{p_k^{3/2}}{\sqrt{N_1}}\right)\right)  \left( P(B_k^{(2)} > F_k) \Exp[\frac{\sigma_k^2}{B_k^{(2)}} | B_k^{(2)} > F_k] + O\left(\frac{1}{N_1}\right) C_k^{\sigma^2} \right) +  O\left(\frac{1}{N_1}\right) + O\left(\frac{\sqrt{N_1}}{N_2^2}\right) \\
    &\leq \sum_{p_k > p_*}^K  \left(O\left(p_k^2\right) + O\left(\frac{p_k^{3/2}}{\sqrt{N_1}}\right)\right)  \left(\Exp[\frac{\sigma_k^2}{B_k^{(2)}} | B_k^{(2)} > F_k] + O\left(\frac{1}{N_1}\right) C_k^{\sigma^2} \right) +  O\left(\frac{1}{N_1}\right) + O\left(\frac{\sqrt{N_1}}{N_2^2}\right) \\
    &\leq \sum_{p_k > p_*}^K  \left(O\left(p_k^2\right) + O\left(\frac{p_k^{3/2}}{\sqrt{N_1}}\right)\right)  \Exp[\frac{\sigma_k^2}{B_k^{(2)}} | B_k^{(2)} > F_k] +  O\left(\frac{1}{N_1}\right) + O\left(\frac{\sqrt{N_1}}{N_2^2}\right) \\
    &\leq \sum_{p_k > p_*}^K  \left(O\left(p_k^2\right) + O\left(\frac{p_k^{3/2}}{\sqrt{N_1}}\right)\right) \frac{\sigma_k^2}{p_kT_kN_2\left(1 - O\left( \sqrt{\frac{1}{p_kN_1}}\right) - O\left( \sqrt{\frac{\ln(1/\gamma)}{N_2p_k^{3/2}}}\right)\right)} +  O\left(\frac{1}{N_1}\right) + O\left(\frac{\sqrt{N_1}}{N_2^2}\right) \\
    &\leq \sum_{p_k > p_*}^K  \left(O\left(\frac{\sqrt{p_k}}{N_2}\right) + O\left(\frac{1}{N_2\sqrt{N_1}}\right)\right) \left(1 + O\left( \sqrt{\frac{1}{p_kN_1}}\right) + O\left( \sqrt{\frac{\ln(1/\gamma)}{N_2p_k^{3/2}}}\right)\right) +  O\left(\frac{1}{N_1}\right) + O\left(\frac{\sqrt{N_1}}{N_2^2}\right) \\
    &\leq \sum_{p_k > p_*}^K  \left(O\left(\frac{\sqrt{p_k}}{N_2}\right) + O\left(\frac{1}{N_2\sqrt{N_1}}\right)\right) \left(1 + O\left( \sqrt{\frac{1}{p_kN_1}}\right) + O\left( \sqrt{\frac{\ln(1/\gamma)}{N_2p_k^{3/2}}}\right)\right) +  O\left(\frac{1}{N_1}\right) + O\left(\frac{\sqrt{N_1}}{N_2^2}\right) \\
    &\leq \sum_{p_k > p_*}^K O\left(\frac{1}{N_2} + \frac{1}{N_2\sqrt{N_1}} + \frac{\sqrt{N_1}}{N_2\sqrt{N_2}}\right) + O\left(\frac{1}{N_2\sqrt{N_1}}\right)\left(1 + O\left( \sqrt{\frac{1}{p_kN_1}}\right) + O\left( \sqrt{\frac{\ln(1/\gamma)}{N_2p_k^{3/2}}}\right)\right) +  O\left(\frac{1}{N_1}\right) + O\left(\frac{\sqrt{N_1}}{N_2^2}\right) \\
    &\leq \sum_{p_k > p_*}^K O\left(\frac{1}{N_2} + \frac{\sqrt{N_1}}{N_2\sqrt{N_2}}\right) + O\left(\frac{1}{N_2\sqrt{N_1}} + \frac{1}{N_2\sqrt{N_1}} + \frac{\sqrt{N_1}}{N_2\sqrt{N_2}} \right) +  O\left(\frac{1}{N_1}\right) + O\left(\frac{\sqrt{N_1}}{N_2^2}\right) \\
    &\leq O\left(\frac{1}{N_1}\right) + O\left(\frac{1}{N_2}\right) + O\left(\frac{\sqrt{N_1}}{N_2\sqrt{N_2}} \right)
\end{align}
}
We will now upper bound the second term in the expression on (90). We apply \Cref{mubound} on $(108)$.
\begin{align}
    \Exp\left[(\sum_{k=1}^K (\hat{w}_k - w_k)\Exp[\hat{\mu_k}] - w_k\mu_k P(B_k=0) )^2\right] &\leq \left[\max_{k} C_k^{\mu_k^2}\right] \Exp\left[\left(\sum_{k=1}^K w_k P(B_k^{(2)} = 0)  + (\hat{w}_k - w_k) \right)^2\right] \\
    &\leq \left[\max_{k} C_k^{\mu_k^2}\right] \Exp\left[\left( O\left(\frac{\sqrt{N_1}}{N_2}\right)  + O\left(\frac{1}{\sqrt{N_1}}\right)\right)^2\right] \\
    &\leq O\left(\frac{1}{N_1}\right) + O\left(\frac{1}{N_2}\right)
\end{align}
Thus, by combining (110) and (113), we finally have that
\begin{align}
    \Exp[(\hmuall - \muall)^2] &= \Exp\left[(\sum_{k=1}^K \hat{w}_k\hat{\mu}_k - w_k\mu_k)^2 \right] \\
    &=  \sum_{k=1}^K \hat{w}_k^2\Var(\hat{\mu}_k) + \left[\max_{k} C_k^{\mu_k^2}\right] \Exp\left[\left(\sum_{k=1}^K w_k P(B_k^{(2)} = 0)  + (\hat{w}_k - w_k)\right)^2\right] \\
    &\leq  \left[O\left(\frac{1}{N_1}\right) + O\left(\frac{1}{N_2}\right) + O\left(\frac{\sqrt{N_1}}{N_2\sqrt{N_2}}\right)\right] + \left[O\left(\frac{1}{N_1}\right) + O\left(\frac{1}{N_2}\right)\right] \\
    &\leq O\left(N_1^{-1}\right) + O\left(N_2^{-1}\right) + O\left(N_1^{1/2}N_2^{-3/2}\right) \qedhere
\end{align}

\end{proof}

\subsection{Sample Reuse}
Note that our algorithm shown in \Cref{pseudocode} involves a parameter $SampleReuse$, which decides whether we reuse samples from Stage 1 in Stage 2 and vice-versa. In our analysis, we assumed that $SampleReuse$ was false. We show that $SampleReuse$ does the change the asymptotic error rates of $\algname$.
\begin{corollary}\label{corrollayx}
    If $SampleReuse$ is true, $\algname$ still achieves the the following upper bound on the the mean squared error for small $\delta > 0$ over Stage 1 and expectation over Stage 2,
    \begin{align}
    \Exp[(\hmuall - \muall)^2] &\leq O\left(N_1^{-1}\right) + O\left(N_2^{-1}\right) + O\left(N_1^{1/2}N_2^{-3/2}\right).
\end{align}
\end{corollary}
\begin{proof}
The high probability bounds we have on $\hat{p}_k$, $\hat{w}_k$, $B_k^{(1)}$, $\hat{\sigma}_k$, $\hat{T}_k$ still hold with sample reuse, so the rates presented in \Cref{finaltheorem} still hold. Furthermore, note that the estimate $\hat{\mu}_k$ is actually a stochastic sum where the number of summands are dependent on Stage 1. As a result, the samples from Stage 1 and Stage 2 are identically and independently distributed, so we can reuse samples when estimating $\hat{\mu}_k$. This will improve the quality of estimates in practice and the bounds we derive for $\hat{p}_k$, $\hat{w}_k$, and $\hat{\mu}_k$ will all be tighter up to constant factors.
\end{proof}

\section{Bootstrap Validity}

\subsection{Validity}
We show that the bootstrap is asymptotically valid.

\minihead{Overview}
We first summarize our approach, which applies the functional delta method.  Suppose that 
$X_1, ..., X_n \sim F_0$ are random variables drawn from a distribution with CDF $F_0$, 
and $\theta$ is a statistical functional of the CDF. If $\theta$ is Hadamard differentiable,
then the functional delta method of \citet[Theorem 23.9]{van2000asymptotic} will imply that
bootstrap resampling gives valid confidence intervals for $\theta$.  We will express our statistic
in this setting, explicitly define $\theta$, and obtain bootstrap validity by showing Hadamard
differentiability of $\theta$.

\minihead{Notation}
We first re-parameterize the draws so that $X_{k, i} = (Y_{k, i}, P_{k, i})$, where $Y_{k, i} \in \mathbb{R}$ is the statistic and $P_{k, i} \in \{ 0, 1 \}$ is the predicate. Namely, $X_{k, i}$ is the joint variable of the statistic and the predicate. Under this parametrization, we can write the $\theta_T$ for any fixed allocation $T$ as 
\begin{align}
    \theta_T(Y_k, P_k) = \frac{1}{\sum_k \Exp[P_k = 1]} \cdot \sum_k \Exp[P_k = 1] \cdot \Exp[Y_k | P_k = 1] f(T_k, T)
\end{align}
for some deterministic function $f$ which normalizes based on the allocation.

For the rest of the proof, let $F_0$ be the CDF of the samples $X_{k, i}$. Let $\hat{F}_n$ be the empirical
CDF and $\hat{F}^*_n$ be the resampled CDF of \algname's bootstrap.  The validity of bootstrap confidence intervals
is given by the following result.

\begin{theorem}
For any fixed allocation $T$, we have the convergence
\begin{align}
    \sqrt{n}(\theta_T(\hat{F}^*_n) - \theta_T(\hat{F}_n)) \to \dot{\theta}_{T, F_0}(\mathbb{B})
\end{align}
to the same distribution $\dot{\theta}_{T, F_0}(\mathbb{B})$ as the rescaled difference $\sqrt{n}(\theta_T(\hat{F}_n) - \theta_T(F_0))$ between the empirical and population statistics. Namely, the bootstrap is valid.
\end{theorem}

\begin{proof}
By \citet[Theorems 23.7 and 23.9]{van2000asymptotic}, it suffices to show that $\theta_T$ is Hadamard differentiable with respect to the CDF.

For any random variable $P_k$ that has mass at 1 such that $P(P_k = 1) > c > 0$ for some constant $c$, we have that $\Exp[P_k = 1]$ and $\Exp[Y_k | P_k = 1]$ are Hadamard differentiable. Thus, $\sum_k \Exp[P_k = 1]$ and $\sum_k \Exp[P_k = 1] \Exp[Y_k | P_k = 1]$ are Hadamard differentiable.

Since $\Exp[P_k = 1]$ is bounded away from 0, by the chain rule (\citet[Theorem 20.9]{van2000asymptotic}), $\frac{1}{\sum_k \Exp[P_k = 1]}$ is also Hadamard differentiable. The Hadamard differentiability of $\theta_T$ follows from the product rule.
\end{proof}

\minihead{Full procedure}
The full procedure is equivalent to mixing samples from Stage 1 and 2. Note that Stage 1 samples are uniformly sampled, so the full procedure is equivalent to samples drawn with weight $\frac{1}{2} (\hat{T}_k + \frac{1}{k})$. Thus, the bootstrap on the full procedure is valid.

\subsection{Practical Considerations}
We note that the bootstrap is \emph{asymptotically} valid, which does not give guarantees in finite sample settings. In particular, if the stratification results in few (e.g., under 30) samples for any given strata, the confidence intervals may be underestimated.

To address this, we can modify the confidence intervals by adjusting them by $\hat{p}_k C^{(\mu)}$ or an upper bound on $p_k$ for the strata with low samples.

\pagebreak

\section{Appendix}

\begin{lemma}[\cite{chung2002connected}]\label{binomialbs}
Let $X_1, ..., X_n$ be independent random variables with:

\begin{align}
    P(X_i=1)=p_i, \ P(X_i=0)=1-p_i
\end{align}

For $X=\sum_{i=1}^n a_iX_i$, we have $\Exp[X] = \sum_{i=1}^n a_ip_i$ and we have define $\nu = \sum_{i=1}^n a_i^2p_i$. Then we have:

\begin{align}
     P(x < \Exp[X] - \lambda) \leq e^{-\lambda^2/2\nu} \\
     P(X > \Exp[X] + \lambda) \leq e^{-\frac{\lambda^2}{2(\nu+a\lambda/3)}}
\end{align}
where $a = max\{a_1, ..., a_n\}$
\end{lemma}

\begin{lemma}[\cite{princetonchernoff}]\label{chernoffbs}
Let $X_1, ..., X_n$ be independent random variables. They do not need to be from the same distribution. Assume that $0 \leq X_i \leq 1$ for each $i$. LEt $X= X_1 + ... + X_n$. Write $\mu = \Exp[X] = \Exp[X_1] + ... + \Exp[X_n]$. Then, for any $\epsilon \geq 0$,

\begin{align}
    P(X \geq (1+e)\mu) \leq exp\left(-\frac{-\epsilon^2}{2+e}\mu\right) \\
    P(X \leq (1-e)\mu) \leq exp\left(-\frac{-\epsilon^2}{2}\mu\right)
\end{align}
\end{lemma}

\begin{lemma}[\cite{ustatistictovariance}]\label{bdtosv}
Let $X_1, ..., X_n$ be i.i.d. For a U-statistic of order 2 in the form $U_n = \frac{1}{\binom{n}{2}} \sum_{i < j} g(X_i, X_j)$, if $g(X_i, X_j)=\frac{1}{2}(X_i - X_j)^2$,
\begin{align}
    \Exp[U_n] = \Var[X]
\end{align}
\end{lemma}

\begin{lemma}[\cite{ustatistic}]\label{bdinq} Let $X_1, ..., X_n$ be i.i.d. For a U-statistic of order 2 in the form $U_n = \frac{1}{\binom{n}{2}} \sum_{i < j} g(X_i, X_j)$ with an upper bound $g(X_i, X_j) \leq b$,
\begin{align}
    P(|U_n - \Exp[U_n]| \geq t) \leq 2\exp(-\frac{nt^2}{8b^2})
\end{align}
\end{lemma}

\bibliography{main}

\begin{thebibliography}{5}
\providecommand{\natexlab}[1]{#1}
\providecommand{\url}[1]{\texttt{#1}}
\expandafter\ifx\csname urlstyle\endcsname\relax
  \providecommand{\doi}[1]{doi: #1}\else
  \providecommand{\doi}{doi: \begingroup \urlstyle{rm}\Url}\fi

\bibitem[Chung and Lu(2002)]{chung2002connected}
Fan Chung and Linyuan Lu.
\newblock Connected components in random graphs with given expected degree
  sequences.
\newblock \emph{Annals of combinatorics}, 6\penalty0 (2):\penalty0 125--145,
  2002.

\bibitem[Ferguson(2005)]{ustatistictovariance}
Thomas Ferguson.
\newblock U-statistics notes for statistics 200c, spring 2005.
\newblock
  \url{http://www.stat.cmu.edu/~arinaldo/Teaching/36710/F18/Scribed_Lectures/Sep24.pdf},
  2005.
\newblock [Online; accessed 11-Feb-2021].

\bibitem[Rinaldo(2018)]{ustatistic}
Alessandro Rinaldo.
\newblock 36-710: Advanced statistical theory lecture 7: September 24.
\newblock
  \url{http://www.stat.cmu.edu/~arinaldo/Teaching/36710/F18/Scribed_Lectures/Sep24.pdf},
  2018.
\newblock [Online; accessed 11-Feb-2021].

\bibitem[Tarjan(2009)]{princetonchernoff}
Robert Tarjan.
\newblock 15-359: Probability and computing lecture 10: More chernoff bounds,
  sampling, and the chernoff + union bound method.
\newblock
  \url{http://aiweb.techfak.uni-bielefeld.de/content/bworld-robot-control-software/},
  2009.
\newblock [Online; accessed 11-Feb-2021].

\bibitem[Van~der Vaart(2000)]{van2000asymptotic}
Aad~W Van~der Vaart.
\newblock \emph{Asymptotic statistics}, volume~3.
\newblock Cambridge university press, 2000.

\end{thebibliography}

\end{document}